\documentclass[english,11pt, a4paper]{article}
\usepackage[T1]{fontenc}
\usepackage[utf8]{inputenc}
\usepackage{geometry}
\usepackage{color}
\usepackage{babel}
\usepackage{float}
\usepackage{mathtools}
\usepackage{enumitem}
\usepackage{amsmath}
\usepackage{amsthm}
\usepackage{amssymb}
\usepackage{graphicx}
\usepackage{wasysym}
\usepackage[numbers]{natbib}
\usepackage{microtype}
\usepackage[unicode=true,
 bookmarks=true,bookmarksnumbered=false,bookmarksopen=false,
 breaklinks=false,pdfborder={0 0 0},pdfborderstyle={},backref=page,colorlinks=true]
 {hyperref}
\hypersetup{pdftitle={Danamic programming},
 pdfauthor={Alois Pichler},
 citecolor=darkgray, urlcolor= darkgray, linkcolor= darkgray}

\makeatletter
\numberwithin{equation}{section}
\theoremstyle{plain}
\newtheorem{thm}{\protect\theoremname}[section]
\theoremstyle{plain}
\newtheorem{lem}[thm]{\protect\lemmaname}
\theoremstyle{definition}
\newtheorem{defn}[thm]{\protect\definitionname}
\theoremstyle{remark}
\newtheorem{rem}[thm]{\protect\remarkname}
\theoremstyle{plain}
\newtheorem{prop}[thm]{\protect\propositionname}
\theoremstyle{definition}
\newtheorem{example}[thm]{\protect\examplename}
\theoremstyle{plain}
\newtheorem{cor}[thm]{\protect\corollaryname}

\mathtoolsset{showonlyrefs}

\usepackage{dsfont}   %	für \one und Erwartungswert   doublestroke
\usepackage{newtxtext, newtxmath}	% Springer like Font
\usepackage{microtype}
\usepackage{siunitx}
\usepackage{tikz}
\usetikzlibrary{matrix,arrows.meta}
\usepackage{doi}

\setlist[enumerate,1]{label=(\roman*), ref=(\roman*)}

\DeclareMathOperator{\E}{{\mathds E}}
\DeclareMathOperator*{\essinf}{ess\,inf}

\newcommand{\one}{{\mathds 1}} 		% \usepackage[sans]{dsfont}

\makeatother

\providecommand{\corollaryname}{Corollary}
\providecommand{\definitionname}{Definition}
\providecommand{\examplename}{Example}
\providecommand{\lemmaname}{Lemma}
\providecommand{\propositionname}{Proposition}
\providecommand{\remarkname}{Remark}
\providecommand{\theoremname}{Theorem}

\begin{document}
\title{\textbf{Foundations of}\\
\textbf{Multistage Stochastic Programming}}
\author{Paul Dommel\thanks{University of Technology, Chemnitz, Faculty of mathematics. 90126
Chemnitz, Germany\protect \\
DFG, German Research Foundation – Project-ID 416228727 – SFB~1410}  \and  Alois Pichler\footnotemark[1]\, \thanks{\protect\includegraphics[width=0.8em]{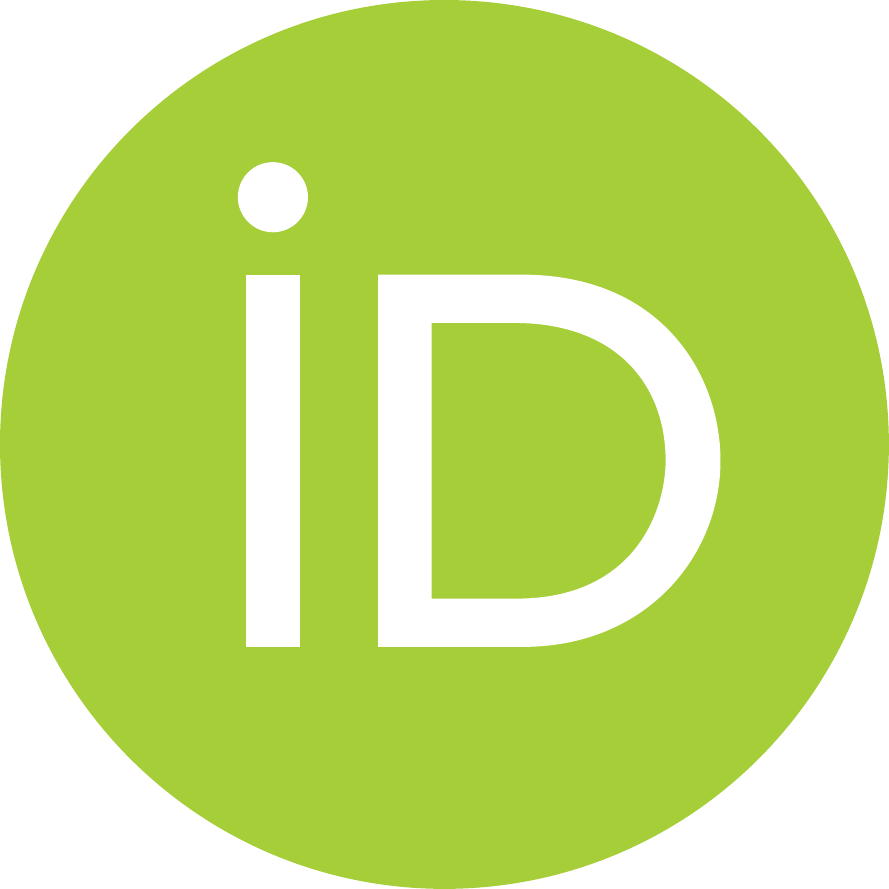}
\protect\href{https://orcid.org/0000-0001-8876-2429}{orcid.org/0000-0001-8876-2429}.
Contact: \protect\href{mailto:alois.pichler@math.tu-chemnitz.de}{alois.pichler@math.tu-chemnitz.de}}}
\maketitle
\begin{abstract}
Multistage stochastic optimization problems are oftentimes formulated
informally in a pathwise way. These are correct in a discrete setting
and suitable when addressing computational challenges, for example.
But the pathwise problem statement does not allow an analysis with
mathematical rigor and is therefore not appropriate.

This paper addresses the foundations. We provide a novel formulation
of multistage stochastic optimization problems by involving adequate
stochastic processes as control. The fundamental contribution is a
proof that there exist measurable versions of intermediate value functions.
Our proof builds on the Kolmogorov continuity theorem. 

A verification theorem is given in addition, and it is demonstrated
that all traditional problem specifications can be stated in the novel
setting with mathematical rigor. Further, we provide dynamic equations
for the general problem, which is developed for various problem classes.
The problem classes covered here include Markov decision processes,
reinforcement learning and stochastic dual dynamic programming. 

\medskip{}


\noindent \textbf{Classification:} 11M35, 11M99, 11Y35 
\end{abstract}

\section{Introduction }

Stochastic optimization problems are frequently considered in finance,
energy management and operations research where it is essential and
of primary interest to develop efficient algorithms and to provide
access to fast decisions. Many of these algorithms build on finite
models in discrete space. Multistage stochastic problems are built
on stochastic processes in discrete time or on decision trees, cf.\ \citet{MaggioniPflug},
\citet{PhilpottMatosFinardi} or \citet{Leclere2014} among many others.

This paper aims at presenting a rigorous mathematical framework for
stochastic optimization problems, particularly multistage stochastic
optimization problems, by systematically exploiting measurability
in stochastic processes, in conditional expectations and by involving
the proper conditional infimum. We develop value processes and show
their relation to the genuine stochastic optimization problem. Our
central result finally resolves measurability of the intermediate
value functions, it builds on the Kolmogorov continuity theorem.

\medskip{}

Multistage stochastic optimization involves optimization based on
partial realizations, which are partially observed trajectories. It
is a major difficulty of multistage stochastic optimization that individual
realizations or trajectories have probability zero. But the problems
are stated naturally in this pathwise way. It is hence essential to
avoid difficulties with arise with this pathwise, or $\omega$\nobreakdash-by\nobreakdash-$\omega$
considerations and to address measurability carefully. 

Early and important attempts to capture measurability are already
present in \citet{rockafellar1976integral} and in \citet{WetsRockafellar97}.
The conditional expectation, the conditional probability and the conditional
infimum constitute main and major difficulties in multistage stochastic
optimization. The infimum in the optimization formulation and the
conditional expectations need to be interchanged at subsequent stages
to exploit computational advantages, cf.\ \citet{Carpentier2015}
or \citet{PflugPichlerBuch}. Indeed, a recourse decision is based
on a partial realization of a stochastic outcome, but has to be considered
already at the very beginning of decision making. Considering every
outcome separately, $\omega$\nobreakdash-by\nobreakdash-$\omega$,
is only possible for finite states, so that a tree describes the evolution
of the stochastic process and the evolution of the decision process
as well. In a multistage environment, however, the computational burden
grows exponentially with the branching structure and this approach
thus is clearly not advisable. The catch phrase \emph{curse of dimensionality}
can be associated with this phenomenon in multistage stochastic optimization.

This paper addresses the general problem of measurability for discrete
and continuous probability measures. The central result is a proof
that there exists a measurable version of the intermediate value process.
We present dynamic equations even for the general, non-Markovian setting.
The general verification theorems presented  are characterizations
as martingales.

We elaborate the theory in full generality and elaborate on problem
settings, which are of particular importance in applications and increasingly
popular in stochastic optimization. They include dynamic programming
(the references \citet{Bertsekas2012} and \citet{Feinberg1996} include
considerations on measurability as well), stochastic dual dynamic
programming, the Bellman principle and reinforcement learning, which
has grown to outstanding importance in machine learning or data science.
For a recent tutorial including also computational aspects we refer
to \citet{Shapiro2020}. 

Investigations on foundations have been started in \citet{PichlerShapiro}
with a focus on the\emph{ distributionally robust} aspect of multistage
stochastic optimization. This paper enhances, complements and continuous
these investigations on foundations, but now addressing the genuine
problem statement itself. 

An important spacial case of multistage stochastic optimization, as
it is presented in this paper, is dynamic optimization. Dedicated
algorithms have been developed for this special case and papers as
\citet{Lan2019} address convergence of dynamic stochastic approximation,
e.g\@., \citet{Carpentier} collect recent theoretical results for
the special case of dynamic optimization again.

Applications of multistage stochastic optimization are widespread
over many economic and managerial disciplines. We pick \citet{Wozabal2013}
to represent and demonstrate importance of multistage stochastic optimization
for example in energy and \citet{Shapiro2013} to exemplify computational
limitations and \citet{Ruszczynski2010} to point to extensions involving
risk.

\paragraph{Outline.}

We address the general multistage problem formulation in Section~\ref{sec:Multistage},
after introducing the informal description and the mathematical setting.
An essential component to manage the evolution of the underlying stochastic
process and the decisions is the value process, introduced in Section~\ref{sec:ValueProcess}.
Particular situations as dynamic problems, additive cost functions,
Markovian processes and SDDP (stochastic dual dynamic programming)
appear frequently in applications. Important simplifications, dedicated
complexity and convergence issues are essential to solve these problems.
We address these particular problem formulations in Section~\ref{sec:Additive}. 

\section{Mathematical setting }

Stochastic optimization builds on random variables, while multistage
stochastic optimization builds on stochastic processes on adequate
probability spaces. In what follows we address the informal, pathwise
setting and then prepare the mathematical stage to discuss the optimization
problem with mathematical rigor. 

\subsection{Informal description}

The multistage optimization problem, stated informally as a work instruction,
is 
\begin{equation}
\inf_{u_{0}}\E_{X_{1}}\dots\E_{X_{t}}\underbrace{\inf_{u_{t}}\underbrace{\E_{X_{t+1}}\inf_{u_{t+1}}\dots\E_{X_{T}}\inf_{u_{T}}v(X_{1},\dots,X_{T},u_{0},\dots u_{T})}_{v_{t}(x_{1:t},u_{0:t})}}_{V_{t}(x_{1:t},u_{0:t-1})}.\label{eq:Symbolic}
\end{equation}
Here, $v$ is the random objective of the optimization problem, $(X_{1},\dots,X_{T})$
are the consecutive random observations and $u_{1},\dots,u_{T}$ the
decisions made after each partial realization $X_{t}$ at each stage~$t$.
The functions $v_{t}$ and $V_{t}$ are the intermediate value functions,
which are given intuitively in~\eqref{eq:Symbolic} in an $\omega$\nobreakdash-by\nobreakdash-$\omega$
or pathwise context. 

The problem statement~\eqref{eq:Symbolic} exhibits the following
difficulties:
\begin{enumerate}[noitemsep,nolistsep]
\item \label{enu:i}The expectation at stage $t$ is a conditional expectation,
conditional on the preceding observations $X_{1},\dots,X_{t}$. This
trajectory has probability~$0$ and the conditional expectation must
not be considered in a pathwise specification  as~\eqref{eq:Symbolic}
does.
\item \label{enu:ii}The infimum with respect to $u_{t}$ at stage $t$
depends on preceding observations. As above, this is a conditional
infimum and not measurable. 
\item \label{enu:iii}The intermediate value functions $v_{t}$ and $V_{t}$
aggregate the entire future. As functions, defined on observed partial
realizations, they are not necessarily measurable.
\end{enumerate}
Nonetheless, the work instruction~\eqref{eq:Symbolic} provides a
straightforward illustration of the optimization problem, indicating
the progression of successive optimization and random realizations.
While~\ref{enu:i} and~\ref{enu:ii} are fixed with standard means,
interchanging the infimum with expectations requires clarification.
The issue~\ref{enu:iii} emerges specifically in multistage optimization.
We resolve this problem with the help of Kolmogorov's continuity theorem.

In what follows we provide a rigorous mathematical problem statement
of~\eqref{eq:Symbolic} first and then discuss derived variants.

\subsection{Mathematical exposition}

Let $(\Omega,\mathcal{F},P)$ be a probability space. We may refer
to \citet[Lemma~1.13]{Kallenberg2002Foundations} or \citet[Theorem~II.4.3]{Shiryaev1996}
for the following Doob–Dynkin lemma.
\begin{lem}[Doob–Dynkin]
\label{rem:DoobDynkin-1}Suppose the random variable $U$ with values
in $\mathbb{R}^{t}$ is measurable with respect to the $\sigma$\nobreakdash-algebra
$\sigma(X)$ generated by the random variable $X$ with values in
$\mathbb{R}^{d}$. Then there is a (Borel\nobreakdash-) measurable
function $\varphi\colon\mathbb{R}^{d}\to\mathbb{R}^{t}$ so that 
\[
U=\varphi\circ X.
\]
\end{lem}

The essential infimum of a set of random variables is defined in \citet{DunfordSchwartz}.
We want to highlight \citet[Appendix~A.5]{Follmer2004} for the most
compelling proof regarding existence.
\begin{defn}[Essential infimum]
Let $\mathcal{U}$ be a family of $\mathbb{R}$\nobreakdash-valued
random variables. The random variable $Y$ is the \emph{essential
infimum} of $\mathcal{U}$, if
\begin{enumerate}[noitemsep]
\item \label{enu:EI1}$Y\le U$ a.e.\ for all $U\in\mathcal{U}$ and 
\item \label{enu:EI2}$Z\le Y$ a.e., whenever $Z\le U$ for all $U\in\mathcal{U}$. 
\end{enumerate}
We shall write $\essinf_{U\in\mathcal{U}}U\coloneqq Y$ for the essential
infimum of $\mathcal{U}$. 

\end{defn}

\begin{rem}
\label{rem:EssInf}The essential infimum exists and is unique, cf.\ \citet[Appendix~A.5]{Follmer2004}
or \citet[Appendix~A]{Karatzas}. If $\mathcal{U}$ is closed under
pairwise minimization,\footnote{The set $\mathcal{U}$ is said to be \emph{directed downwards} in
\citet{Follmer2004}.} i.e, $\min\left(U,\,V\right)\in\mathcal{U}$ for $U$, $V\in\mathcal{U}$,
then there is a nonincreasing sequence $U_{n}\in\mathcal{U}$ such
that $U_{n}\to\essinf_{U\in\mathcal{U}}U$ a.s., as $n\to\infty$.
\end{rem}

For $X$ measurable, the random variable $\essinf_{U\in\mathcal{U}}\bigl(\left.U\right|\sigma(X)\bigr)$
is measurable with respect to $\sigma(X)$, the sigma algebra generated
by $X$. By the Doob–Dynkin lemma there is a measurable $\varphi(\cdot)$
so that $\essinf_{U\in\mathcal{U}}\bigl(\left.U\right|\sigma(X)\bigr)=\varphi(X)$.
We shall denote this function by $\essinf_{U\in\mathcal{U}}\left(\left.U\right|X\right)\coloneqq\varphi$.
\begin{rem}
We shall also address the conditional essential infimum for a singleton
$\mathcal{U}=\left\{ U\right\} $. In this case, the random variable
$\essinf_{U\in\mathcal{U}}\left(\left.U\right|X\right)$ is the $\sigma(X)$\nobreakdash-measurable
envelope of $U$ for which we shall write $\essinf\left(\left.U\right|X\right)$. 
\end{rem}

\begin{rem}[Caveat]
The term essential infimum is occasionally also used for the largest
\emph{number} $c\in\mathbb{R}$ smaller than the random variable~$X$,
$c\le X$ a.s\@. This is $\essinf_{U\in\mathcal{U}}\left(\left.U\right|\left\{ \emptyset,\Omega\right\} \right)$
in the notation introduced, where $\left\{ \emptyset,\Omega\right\} $
is the trivial sigma algebra. 
\end{rem}

\subsection{Functional optimization}

The prevailing perspective in practice of multistage stochastic optimization
is not a measure theoretic perspective but rather a functional view:
we shall develop and address this perspective as the \emph{informal},
\emph{$\omega$\nobreakdash-by\nobreakdash-$\omega$} or \emph{pathwise
description}. Throughout, we will give the stochastic process perspective
first and then complement the informal perspective as well. While
the first one provides expressions with mathematical rigor, the latter,
intuitive problem statement is perhaps better to understand, well-established
and more practical for concrete numerical implementations. This is
essential for both, the governing stochastic process and the decision
process. 
\begin{defn}[Decomposable functions]
\label{def:Decomposable}Let $\sigma(\mathcal{U})\coloneqq\sigma(u\colon u\in\mathcal{U})$
be the sigma algebra generated by the functions $u\colon\mathbb{R}^{t}\to\mathbb{R}^{d}$
contained in~$\mathcal{U}$. We shall say that the class of functions~$\mathcal{U}$
is \emph{decomposable}, if $u_{A}\in\mathcal{U}$, where 
\[
u_{A}(x)\coloneqq\begin{cases}
u_{1}(x) & \text{if }x\in A,\\
u_{2}(x) & \text{else}
\end{cases}
\]
whenever $A\in\sigma(\mathcal{U})$ and $u_{1}$, $u_{2}\in\mathcal{U}$. 
\end{defn}

Traditional formulations of the interchangeability principle require
that the infimum is measurable (cf.\ \citet{ShapiroInterchangeability}
or the normal integrands in \citet[Theorem~14.60]{WetsRockafellar97}
or \citet{Rockafellar1982}). By involving the essential infimum,
the following proposition establishes the interchangeability principle
without requesting measurability explicitly.
\begin{prop}[Interchangeability principle]
\label{prop:Interchange}Let $\mathcal{U}$ be a class of measurable
functions, let $v\colon\mathbb{R}^{t}\times\mathbb{R}^{d}\to\mathbb{R}$
be a (measurable) function bounded from below and $X\colon\Omega\to\mathbb{R}^{t}$
a random variable. It holds that 
\begin{equation}
\E\essinf_{u\in\mathcal{U}}v\big(X,u(X)\big)\le\inf_{u\in\mathcal{U}}\E v\big(X,u(X)\big).\label{eq:inf}
\end{equation}
Equality holds in~\eqref{eq:inf} if $\mathcal{U}$ is decomposable
and $X$ is measurable with respect to $\sigma(\mathcal{U})$. 
\end{prop}

\begin{proof}
For every $x$ we have that $\inf_{u\in\mathcal{U}}v\big(x,u(x)\big)\le v\bigl(x,u(x)\bigr)$
and thus $\essinf_{u\in\mathcal{U}}v\big(X,u(X)\big)\le v\big(X,u(X)\big)$
a.e\@. Taking expectations first and then the infimum reveals~\eqref{eq:inf}.

For the remaining assertion recall from Remark~\ref{rem:EssInf}
(or \citet[Appendix~A]{Karatzas}) that there is a sequence $u_{j}$
so that $\min_{j=1,\dots,n}v\big(X,u_{j}(X)\big)\to\essinf_{u\in\mathcal{U}}v\big(X,u(X)\big)$
almost surely, as $n\to\infty$. Define 
\[
A_{i}\coloneqq\left\{ v\big(X,u_{i}(X)\big)=\min_{j=1,\dots,n}v\big(X,u_{j}(X)\big)\right\} ,\qquad\tilde{A}_{i}:=A_{i}\backslash\bigcup_{j<i}A_{j}
\]
and set $\tilde{u}_{n}\coloneqq\sum_{i=1}^{n}u_{i}\cdot\one_{\tilde{A}_{i}}$.
As $\mathcal{U}$ is decomposable we have that $\tilde{u}_{n}\in\mathcal{U}$
and $v\big(x,\tilde{u}_{n}(x)\big)=\min_{i=1,\dots,n}v\big(x,u_{i}(x)\big)$.
Employing Beppo Levi's monotone convergence theorem we conclude that
$\E v\big(X,\tilde{u}_{n}(X)\big)\to\E\essinf_{u\in\mathcal{U}}v\big(X,u(X)\big)$
as $n\to\infty$ and hence the assertion. 
\end{proof}
\begin{prop}
Suppose that $u\mapsto v(x,u)$ is monotone for every $x$, i.e.,
$v(x,u_{1})\le v(x,u_{2})$ whenever $u_{1}\le u_{2}$ in every component
and $\min(u_{1},u_{2})\in\mathcal{U}$ for $u_{1}$, $u_{2}\in\mathcal{U}$.
Then interchangeability~\eqref{eq:inf} holds with equality. 
\end{prop}

\begin{proof}
By monotonicity of $v$ we have with $u\coloneqq\min_{i=1,\dots,n}u_{i}\in\mathcal{U}$
that 
\[
\min_{j=1,\dots,n}v\big(X,u_{j}(X)\big)=v\big(X,\min_{j=1,\dots,n}u_{j}(X)\big)=v\big(X,u(X)\big).
\]
The assertion follows along the proof of Proposition~\ref{prop:Interchange}.
\end{proof}

\section{General multistage optimization problems\label{sec:Multistage}}

The general multistage optimization problem involves a stochastic
process instead of a simple random variable. Let $X=(X_{1},\dots,X_{t})$
be a stochastic process with stages $t=1,\dots,T$ and, without loss
of generality, with marginals $X_{t}\in\mathbb{R}$. For convenience,
the stochastic process~$X$ is occasionally also augmented with a
deterministic starting value $X_{0}=x_{0}$ a.s.\ so that $X=(X_{0},X_{1},\dots,X_{T})$.
\begin{defn}[Nonanticipativity]
The stochastic process $U=(U_{0},\dots,U_{T})$ is \emph{adapted}
to~$X$, if~$U_{t}$ is measurable with respect to $\sigma(X_{0},\dots,X_{t})$
for every $t=0,\dots,T$. We shall write 
\[
U\lhd X,
\]
if $U$ is adapted to $X$.
\end{defn}

In stochastic optimization, the synonymous term \emph{nonanticipative}
is more common than adapted.
\begin{defn}[The natural filtration]
The stochastic process $X=(X_{0},\dots,X_{T})$ is adapted to the
\emph{natural filtration}, if $X_{t}(\omega)=X_{t}(\omega_{1},\dots,\omega_{t})$
(that is, $X_{t}(\omega)=\tilde{X}_{t}(\omega_{1},\dots,\omega_{t})$
for some random variable $\tilde{X}$ which we identify with $X_{t}$). 
\end{defn}

Multistage stochastic optimization considers classes $\mathcal{U}$
of stochastic control processes. To not run into difficulties regarding
a governing measure we assume that there is a control $U_{0}$ so
that 
\[
U\lhd U_{0}\text{ (nonanticipative) for all }U\in\mathcal{U}.
\]
A particular situation arises for the class $\mathcal{U}$ of stochastic
processes adapted to $X$, $\mathcal{U}\subset\left\{ U\colon U\lhd X\right\} $.
In this case one may chose $U_{0}=X$ as governing process. 

We consider the following, general multistage stochastic optimization
problem.
\begin{defn}[Multistage optimization problem]
Let 
\begin{align}
v\colon\mathbb{R}^{T+1}\times\mathbb{R}^{T+1} & \to\mathbb{R}\label{eq:Value}\\
(x,u) & \mapsto v(x,u)\nonumber 
\end{align}
be a measurable function. For a class $\mathcal{U}$ of feasible controls,
the general multistage stochastic optimization problem is 
\begin{equation}
\inf_{\substack{U\in\mathcal{U},\\
U\lhd X
}
}\E v(X,U),\label{eq:v0}
\end{equation}
where the infimum is among all feasible control policies $U\in\mathcal{U}$
adapted to $X$. The function~$v$ is the (stochastic) \emph{objective
function} and the set $\mathcal{U}$ is the set of \emph{admissible
controls}, \emph{decisions} or \emph{policies}. Note that the decision
space is $\mathbb{R}^{T+1}$ in~\eqref{eq:Value}, that is, at each
stage $t\in\left\{ 0,\dots,T\right\} $ a decision in $\mathbb{R}$
is made; this setting is chosen for convenience of presentation. 
\end{defn}

In what follows we shall assume that the infimum in~\eqref{eq:v0}
is finite. A somewhat stronger assumption, although not necessary,
is that $v$ is uniformly bounded from below (i.e., $v\ge C>-\infty$)
so that the expectation in~\eqref{eq:v0} is well-defined for every
$U\in\mathcal{U}$.

\subsection{Equivalent problem statements}

For $u_{t}(x_{1},\dots,x_{t})$ measurable it is evident that $u_{t}(X_{1},\dots,X_{t})$
is measurable with respect to $\sigma(X_{1},\dots,X_{t})$. For this,
\begin{equation}
U\coloneqq u(X_{1},\dots,X_{T})\label{eq:Function}
\end{equation}
 is a nonanticipative process with respect to $X$, provided that
\begin{equation}
u(x_{1},\dots,x_{T})=\begin{pmatrix}\begin{array}{l}
u_{0}\\
u_{1}(x_{1})\\
\qquad\vdots\\
u_{T}(x_{1},\dots,x_{T})
\end{array}\end{pmatrix}.\label{eq:u}
\end{equation}
The Doob–Dynkin lemma (Lemma~\ref{rem:DoobDynkin-1}) ensures that
every process $U\in\mathcal{U}$ adapted to~$X$ has the particular
form~\eqref{eq:Function} with~\eqref{eq:u}. 
\begin{lem}[Doob–Dynkin lemma, extended]
\label{lem:DoobDynkin}Let $X=(X_{0},\dots,X_{T})$ be a stochastic
process in discrete time with marginals states $X_{t}\in\mathbb{R}^{d}$
and $U\lhd X$. There are measurable functions $u_{t}$ so that $U_{t}=u_{t}(X_{1},\dots,X_{t})$
for $t=0,\dots,T$ a.s.\ and $U=\varphi_{U}\circ X$, where $\varphi_{U}=u$
is given by~\eqref{eq:u}. 
\end{lem}

\paragraph{Functional optimization perspective.}

The optimization problem~\eqref{eq:v0} employs a fixed stochastic
process~$X$. In view of the Doob–Dynkin lemma, the  problem~\eqref{eq:v0}
thus can be stated as an optimization problem among stochastic processes,
or equivalently also as optimization problem among functions, each
of the specific form~\eqref{eq:u}. The multistage stochastic optimization
problem thus can be classified as a \emph{functional optimization
problem}, because solving it means finding unknown \emph{functions}
as~\eqref{eq:u}. The equivalence between measurable functions and
processes is given by 
\[
U\mapsto\varphi_{U},
\]
where $\varphi_{U}$ is the function from the extended Doob–Dynkin
lemma (Lemma~\ref{lem:DoobDynkin}), while the inverse is the map
\[
u\mapsto U=u(X)
\]
 given in~\eqref{eq:Function}.

Further, this equivalence allows extending the notion of decomposable
to stochastic processes.
\begin{defn}[Decomposable processes]
The class $\mathcal{U}$ of stochastic process is decomposable, if
each function in
\[
\left\{ \varphi_{U}\colon U\in\mathcal{U}\right\} 
\]
is decomposable in the sense of Definition~\ref{def:Decomposable}. 
\end{defn}

\subsection{Special cases of the general problem setting}

The conventional stochastic optimization problem and the stochastic
optimization problem with recourse are special cases of the multistage
stochastic optimization problem.
\begin{example}[$T=0$]
Consider the set of policies with $\mathcal{U}\subset\left\{ U\colon U_{t}\lhd X_{0}\text{ for all }t\ge0\right\} $
(or $T=0$), so that each component $u_{t}$ is deterministic, i.e.,
nonrandom. The corresponding optimization problem 
\begin{equation}
\inf_{u\in\mathbb{R}^{T+1}}\E v(X,u)\label{eq:Classical-1}
\end{equation}
is a conventional stochastic optimization problem, as it is sufficient
to treat $X$ as a random vector in~\eqref{eq:Classical-1}. Here,
it is not essential that $X$ is a stochastic process, the time component
is missing.
\end{example}

\begin{example}[$T=1$]
Consider the feasible policies 
\[
\mathcal{U}\subset\left\{ u\colon u_{0}\lhd X_{0}\text{ and }u_{t}\lhd X_{1}\text{ for all }t\ge1\right\} 
\]
(or $T=1$). With~\eqref{eq:u}, the problem simplifies to 
\begin{equation}
\inf_{(u_{0},u_{1}(X))\in\mathcal{U}}\E v\big(X_{1},u_{0},u_{1}(X_{1})\big).\label{eq:Classical}
\end{equation}
Here, the decision $u_{0}$ is deterministic, i.e., does not depend
on the random components of $X$; $u_{1}(\cdot)$ is called the \emph{random
recourse decision} in the literature (cf.\ \citet{RuszczynskiShapiro2009}).
\end{example}

\section{The value process\label{sec:ValueProcess}}

It is an important conceptual element in stochastic optimization to
consider the problem sequentially in time, so that any new observation
$X_{t}$ triggers a subsequent new decision $u_{t}(X_{1}\dots,X_{t})$,
which itself is based on the past. \citet{Shapiro2012} depicts the
consecutive transitions via the chain in Figure~\ref{fig:Progress}.
The transitions Figure~\ref{fig:Progress} can be started with $X_{0}$
equally well.

\begin{figure}[H]
\[
u_{0}\leadsto X_{1}\leadsto u_{1}\leadsto\dots\leadsto X_{t}\leadsto u_{t}\leadsto X_{t+1}\leadsto\dots\leadsto X_{T}\leadsto u_{T}
\]

\caption{The progression of random observations and decisions\label{fig:Progress}}
\end{figure}

In what follows we develop a similar decomposition of the optimization
problem~\eqref{eq:v0} and present our main result in Theorem~\ref{thm:21}
below. For notational convenience we introduce the abbreviation $x_{t:t^{\prime}}\coloneqq(x_{t},x_{t+1},\dots,x_{t^{\prime}})$
($0\le t$, $t^{\prime}\le T$) for subvectors. We also write $X_{:t}\coloneqq(X_{0},\dots,X_{t})$
for the initial and $U_{t:}\coloneqq(U_{t},\dots,U_{T})$ for the
final (trailing) substrings. Recall that $U$ is a non-anticipative
process if there is a control $u$ so that $U=u(X_{1,}\dots,X_{T})$,
as well as a functions $u_{t:}$ with $U_{t:}=u_{t:}(X)$. By $\mathcal{U}_{t:}=\left\{ u_{t:}\colon U\in\mathcal{U}\right\} $
we denote the set of functions including the final decisions of all
control processes.

\subsection{Existence of the intermediate value functions}

A common way to solve the initial problem~\eqref{eq:v0} is to decompose
it into a sequence of subproblems. We specify these subproblems by
introducing the value process in the following considerations. Let
$u_{:t}\in\mathbb{R}^{t+1}$ and a function $\tilde{u}_{t+1:T}\in\mathcal{U}_{t+1:T}$
be given. As a consequence of the Doob–Dynkin lemma (Lemma~\ref{rem:DoobDynkin-1})
there is measurable mapping $v_{t,u_{:t}}^{\tilde{u}_{t+1:T}}\colon\mathbb{R}^{t+1}\to\mathbb{R}$
such that 
\begin{equation}
v_{t,u_{:t}}^{\tilde{u}_{t+1:T}}(X_{:t})=\E\big(\left.v\left(X,u_{:t},\tilde{u}_{t+1:T}(X)\right)\right|X_{:t}\big).\label{eq:V4}
\end{equation}
These conditional expectations constitute the building block for the
intermediate value functions.
\begin{defn}
\label{def:IntermediateValue}The \emph{(intermediate) value functions}
are 
\begin{align}
v_{t}(x_{:t},u_{:t}) & \coloneqq\essinf_{\tilde{u}_{t+1:T}\in\mathcal{U}_{t+1:T}}v_{t,u_{:t}}^{\tilde{u}_{t+1:T}}\left(x_{:t}\right)\text{ and}\label{eq:v3}\\
V_{t}(x_{:t},u_{:t-1}) & \coloneqq\essinf_{\tilde{u}_{t}\in\mathcal{U}_{t}}v_{t}\left(x_{:t},u_{:t-1,}\tilde{u}_{t}(x_{:t})\right),\label{eq:V3}
\end{align}
where $t=0,\dots,T$.
\end{defn}

These value functions are functions on $\mathbb{R}^{(t+1)\times(t+1)}$
($\mathbb{R}^{(t+1)\times t}$, resp.)\ and the essential imfima
are with respect to these spaces. These functions are generally not
unique as there are multiple functions satisfying the Doob–Dynkin
lemma. The value functions~$V_{t}$ and~$v_{t}$ are defined pointwise
(and well-defined on each point), but they are not necessarily measureable.
Hence, additional conditions on~$v$ need to be imposed to ensure
measurability.

\medskip{}

The following statement is the main result. It establishes existence
of a measurable version of the intermediate value functions. The proof
builds on \emph{Kolmogorov's continuity theorem}, also known as Kolmogorov–Chentsov
theorem.
\begin{thm}[Existence of a measurable version of the value function]
\label{thm:21}Assume that $v$ is locally Hölder continuous with
exponent $\alpha>0$ in~$u$, i.e., 
\begin{equation}
\left|v(x,u_{1})-v(x,u_{2})\right|\le C\left\Vert u_{1}-u_{2}\right\Vert ^{\alpha}\qquad\text{for }x\in\mathbb{R}^{t}\ \text{ and }\ \left\Vert u_{1}-u_{2}\right\Vert \le\delta,\label{eq:Hoelder}
\end{equation}
where $\delta>0$ is sufficiently small. Then there exists a version
of~$v_{t}$ of the intermediate value function which is measurable
with respect to $\mathcal{B}(\mathbb{R}^{t+1})\otimes\mathcal{B}(\mathbb{R}^{t+1})$
and locally Hölder continuous with exponent $\tilde{\alpha}\in\left(0,\frac{\alpha}{t+2}\right)$.
\end{thm}

To prove the main theorem we recall the following condition on joint
measureability from \citet[Theorem~2]{Gowrisankaran1972}; we state
the result in full mathematical beauty, although we do not need this
most general variant.
\begin{thm}
\label{thm:22}Let $(X,\tau)$ be a measureable space and $Y$ a Suslin
space. Let $\mathcal{B}$ be the Borel $\sigma$\nobreakdash-algebra
of all measurable subsets for a locally finite measure $\lambda$
on the Borel $\sigma$-algebra of $Y$. Then, a function $f\colon X\times Y\to A$
with values in a separable metrizable space $A$ with
\begin{enumerate}[noitemsep,nolistsep]
\item \label{enu:1}$x\mapsto f(x,y)$ is $\tau$-measurable for every
$y\in Y$ and
\item \label{enu:2}$y\mapsto f(x,y)$ is continuous on $Y$ for each $x\in X$
\end{enumerate}
is $\tau\otimes\mathcal{B}$-measurable on $X\times Y$.

\end{thm}

\begin{rem}
Functions satisfying the conditions~\ref{enu:1} and~\ref{enu:2}
of Theorem~\ref{thm:22} are also known  as \emph{Carathéodory functions}. 
\end{rem}

\begin{proof}[Proof of Theorem~\ref{thm:21}]
We shall employ Theorem~\ref{thm:22}. Consider the function 
\[
v_{t}^{\tilde{u}_{t+1:T}}(x_{:t},u_{:t})\coloneqq v_{t,u_{:t}}^{\tilde{u}_{t+1:T}}(x_{:t})
\]
(cf.~\eqref{eq:V4}), where $\tilde{u}_{t+1:T}\in\mathcal{U}_{t+1:T}$
is fixed. Measurability follows from the definition of the function~$v_{t}$
in~\eqref{eq:v3} and general measurability of the essential infimum
and thus the condition~\ref{enu:1} of Theorem~\ref{thm:22}. 

It remains to verify continuity, i.e.,~\ref{enu:2}. In order to
employ Theorem~\ref{thm:22} we need to show continuity of $u_{:t}\mapsto v_{t}^{\tilde{u}_{t+1:T}}\left(x_{:t},u_{:t}\right)$.
To this end consider the stochastic process $\left(Z_{u}\right)_{u\in\mathbb{R}^{t+1}}$,
indexed by $u\in\mathbb{R}^{t+1}$ and defined by 
\[
Z_{u}\coloneqq\E\bigl(\left.v\left(X,u,\tilde{u}_{t+1.T}(X)\right)\right|X_{:t}\bigr).
\]
Further, let $u_{0}\in\mathbb{R}$ and $u\in U_{\delta}(u_{0})$ for
$\delta>0$ sufficiently small be given. Set $\tilde{\alpha}\coloneqq\frac{t+2}{\alpha}$
, $\beta\coloneqq1$ and by employing the Hölder condition~\eqref{eq:Hoelder}
we have that 
\begin{align*}
\E\left|Z_{u}-Z_{u_{0}}\right|^{\tilde{\alpha}} & =\E\big|\E\left(\left.v\left(X,u,\tilde{u}_{t+1.T}(X)\right)-v\left(X,u_{0},\tilde{u}_{t+1.T}(X)\right)\right|X_{:t}\right)\big|^{\tilde{\alpha}}\\
 & \le\E\left(C\cdot\left\Vert u-u_{0}\right\Vert ^{\alpha}\right)^{\tilde{\alpha}}\\
 & \le C\left\Vert u-u_{0}\right\Vert ^{t+2}=C\left\Vert u-u_{0}\right\Vert ^{t+1+\beta}
\end{align*}
for some $C<\infty$. Hence, by the Kolmogorov continuity theorem
(cf.\ \citet[p.~453]{Klenke}), there is a process $Z_{u}$ such
that $Z_{u}=Z(\cdot,u)=\E\left(v\left(X,u,\tilde{u}_{t+1:T}(X)\right)|X_{:t}\right)$
and $Z(\omega,\cdot)$ is Hölder continuous with exponent $\frac{\beta}{\tilde{\alpha}}=\frac{\alpha}{t+2}$
for almost every $\omega\in\Omega$. It follows that the corresponding
functions $v_{t}^{\tilde{u}_{t+1:T}}$ are continuous with respect
to $u$. This proves~\ref{enu:2} and hence the assertion of the
theorem.
\end{proof}
\begin{rem}[Lipschitz continuity]
It is evident that measurable versions of~\eqref{eq:v3} and~\eqref{eq:V3}
exist for uniformly Lipschitz continuous objective functions~$v$.
\end{rem}

\subsection{The value process}

In what follows, we define the value processes substituting $x_{:t}$
and $u_{:t}$ by their stochastic counterparts~$X_{:t}$ and $U_{:t}$.
\begin{defn}[Value process]
\label{def:Value-1-1} Assume $v$ satisfies the Hölder condition
imposed in Theorem~\ref{thm:21} and $U\in\mathcal{U}$ is a nonanticipative
stochastic process ($U\lhd X$). The \emph{general value processes}
are 
\begin{align}
\boldsymbol{v}_{t}^{U} & \coloneqq v_{t}(X_{:t},U_{:t})\text{ and}\label{eq:v-4}\\
\boldsymbol{V}_{t}^{U} & \coloneqq V_{t}(X_{:t},U_{:t-1}),\label{eq:V-4}
\end{align}
where $v_{t}$ and $V_{t}$ are the intermediate value functions,
cf.\ Definition~\ref{def:IntermediateValue}.
\end{defn}

\begin{rem}
The functions $v_{t}$ and $V_{t}$ (cf.~\eqref{eq:v3} and~\eqref{eq:V3})
are defined on
\[
V_{t}\colon\mathbb{R}^{t+1}\times\mathbb{R}^{t}\to\mathbb{R}\text{ and }v_{t}\colon\mathbb{R}^{t+1}\times\mathbb{R}^{t+1}\to\mathbb{R}.
\]
We employ bold letters to indicate random variables, i.e., functions
on $\Omega$ given by 
\begin{align*}
\boldsymbol{v}_{t}^{U}(\omega) & =v_{t}\bigl(X_{:t}(\omega),U_{:t}(\omega)\big)\text{ and}\\
\boldsymbol{V}_{t}^{U}(\omega) & =V_{t}\big(X_{:t}(\omega),U_{:t-1}(\omega)\big).
\end{align*}
Figure~\ref{fig:Diagram} depicts the domain and the range of these
functions and random variables. 
\end{rem}

\begin{figure}
\begin{centering}
\begin{tikzpicture}
\matrix (m)[matrix of math nodes, row sep= 3em, column sep= 5em]
  {\Omega              & & \mathbb R\\
& \mathbb R^{T+1}\times\mathbb R^{T+1} & \\
& \mathbb R^{t+1}\times\mathbb R^{t+1} & \\};
\path (m-1-1) edge [->] node [below left] {$(X_{:t},U_{:t})$} (m-3-2);
\path (m-1-1) edge [->] node [above] {$\boldsymbol v_t$} (m-1-3);
\path (m-1-1) edge [->] node [above] {$(X,U)$} (m-2-2);
\path (m-2-2) edge [->] node [right] {} (m-3-2);
\path (m-2-2) edge [->] node [above] {$v$} (m-1-3);
\path (m-3-2) edge [->] node [below right] {$v_t\le \E\big(\left.v(X,U)\right|X_{:t},U_{:t}\big)$} (m-1-3);
\end{tikzpicture}
\par\end{centering}
\caption{Domain and range of the objective and the general value process\label{fig:Diagram}}
\end{figure}
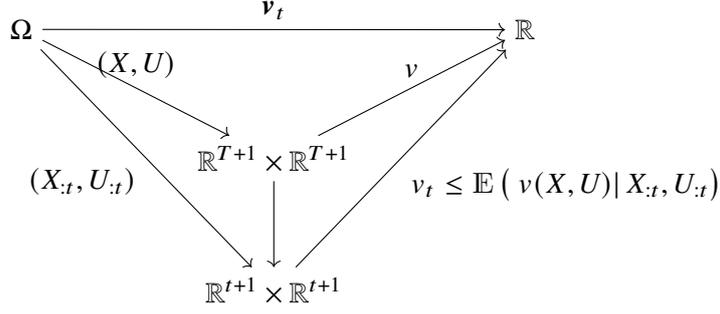

\begin{rem}[Pathwise, or $\omega$\nobreakdash-by\nobreakdash-$\omega$ description]
\label{def:Value}The functions $V_{t}$ and $v_{t}$ describing
the value processes~\eqref{eq:v-4} and~\eqref{eq:V-4} can be given
explicitly and directly—but intuitively—as 
\begin{equation}
V_{t}(x_{:t},u_{:t-1})=\inf_{u_{t:}(\cdot)}\E\left(\left.v\bigl(X_{:T},u_{:t-1},u_{t:T}(X_{:T})\bigr)\right|X_{:t}=x_{:t},U_{:t}=u_{:t-1}\right)\label{eq:v-3}
\end{equation}
and
\begin{equation}
v_{t}(x_{:t},u_{:t})=\inf_{u_{t+1:}(\cdot)}\E\left(\left.v\bigl(X_{:T},u_{:t},u_{t+1:T}(X_{:T})\bigr)\right|X_{:t}=x_{:t},U_{:t}=u_{:t}\right),\label{eq:V-3}
\end{equation}
where the infima are among functions 
\[
u_{t:}(x_{1},\dots,x_{T})=\begin{pmatrix}\begin{array}{l}
u_{t}(x_{1},\dots,x_{t})\\
\qquad\quad\vdots\\
u_{T}(x_{1},\dots,x_{t},\dots,x_{T})
\end{array}\end{pmatrix}
\]
with $u(X)\in\mathcal{U}$.

Note, however, that the expressions~\eqref{eq:v-3} and~\eqref{eq:V-3}
are not necessarily well defined, as they may depend explicitly on
the choice of the control process $U_{:t}$. They further face a delicate
measurability problem, as the pointwise infimum is not measurable,
in general. Hence~\eqref{eq:v-3} and~\eqref{eq:V-3} can\emph{not
}be used as definitions. Our definitions~\eqref{eq:v3} and~\eqref{eq:V3},
together with~\eqref{eq:v-4} and~\eqref{eq:V-4}, resolve this
problem by addressing $u_{:t}$ as a parameter and passing over to
the essential infimum, which has a measurable version by the main
theorem, Theorem~\ref{thm:22}. 
\end{rem}

\subsection{Relation to the multistage problem}

In what follows we derive the equations interconnecting the value
functions introduced in the preceding section. To this end observe
first that 
\begin{equation}
V_{0}=\inf_{U\in\mathcal{U}}\E v(X,U)\label{eq:V0inf}
\end{equation}
by definition~\eqref{eq:V3}, so that $V_{0}$ is the optimal value
of the initial problem. Further, we have with~\eqref{eq:v3} that
\begin{equation}
v_{T}=v,\label{eq:7-2}
\end{equation}
which is the starting point of the optimization problem at the final
stage.\medskip{}

The following statements interconnect the value functions at intermediate
stages.
\begin{thm}
\label{thm:8}Let $U\in\mathcal{U}$ be a feasible policy. It holds
that 
\begin{align}
V_{t}(X_{:t},U_{t-1}) & =\essinf_{\tilde{u}_{t}\in\mathcal{U}_{t}}v_{t}(X_{:t},U_{:t-1},\tilde{u}_{t}(X_{:t}))\text{ and}\label{eq:18}\\
v_{t}(X_{:t},U_{:t}) & \ge\E\left(\left.V_{t+1}\left(X_{:t+1},U_{:t}\right)\right|X_{:t}\right).\label{eq:17}
\end{align}
Equality holds in~\eqref{eq:17}, if $\mathcal{U}$ is decomposable.
\end{thm}

\begin{proof}
The first equation follows directly from the definition of $V_{t}$
and $v_{t}$. The second follows from 
\begin{align*}
\E\big(\left.V_{t+1}\left(X_{:t+1},U_{:t}\right)\right|X_{:t}\big) & =\E\left(\left.\essinf_{\text{\ensuremath{\tilde{u}_{t+1:T}}}\in\mathcal{U}_{t+1:T}}\E\left(\left.v(X,U_{:t},\tilde{u}_{t+1:T}(X))\right|X_{:t+1}\right)\right|X_{:t}\right)\\
 & \le\essinf_{\text{\ensuremath{\tilde{u}_{t+1:T}}}\in\mathcal{U}_{t+1:T}}\E\big(\left.\E\left(\left.v(X,U_{:t},\tilde{u}_{t+1:T}(X))\right|X_{:t+1}\right)\right|X_{:t}\big)\\
 & =\essinf_{\text{\ensuremath{\tilde{u}_{t+1:T}}}\in\mathcal{U}_{t+1:T}}\E\big(\left.v(X,U_{:t},\tilde{u}_{t+1:T}(X))\right|X_{:t}\big)
\end{align*}
by Proposition~\ref{prop:Interchange} and the tower property of
the conditional expectation. Equality holds, by Proposition~\ref{prop:Interchange}
again, for decomposable controls and hence the assertion. 
\end{proof}
\begin{rem}[Pathwise, or $\omega$\nobreakdash-by\nobreakdash-$\omega$ description]
As above and employing the functions~\eqref{eq:v-3} and~\eqref{eq:V-3},
the equations can be stated directly and explicitly by 
\begin{align*}
V_{t}(x_{:t},u_{:t-1}) & =\inf_{u_{t}}v_{t}(x_{:t},u_{:t-1},u_{t})\quad\text{ and}\\
v_{t}(x_{:t},u_{:t}) & \ge\E_{X_{t+1}}\left(\left.V_{t+1}(X_{:t+1},u_{:t})\right|X_{:t}=x_{:t}\right)\\
 & =\E_{X_{t+1}}\left(\left.V_{t+1}(x_{:t},X_{t+1},u_{:t})\right|X_{:t}=x_{:t}\right).
\end{align*}
Equality holds, if $\mathcal{U}$ is decomposable.

These equations get to the point directly and explain the computational
task at each stage ($t)$ and at each node ($x_{:t}$, $u_{:t-1}$).
Note again that stating the equations this way is not justified from
a mathematical perspective, the equations suffer from measurability
issues, in general. They are justified in the finite dimensional case
if $P(X_{:t}=x_{:t}\text{ and }U_{:t-1}=u_{t-1})>0$. 
\end{rem}

The mutual relations above give rise to combining the components to
the following dynamic equations.
\begin{cor}[Dynamic relations]
\label{cor:Dynamic}Let $U\in\mathcal{U}$ be a feasible control
process. It holds that 
\begin{align*}
\boldsymbol{V}_{t}^{U} & \ge\essinf_{U_{:t}^{\prime}\in\mathcal{U}_{:t},\,U_{:t-1}^{\prime}=U_{:t-1}}\E\left(\left.\boldsymbol{V}_{t+1}^{U^{\prime}}\right|X_{:t}\right)\quad\text{ and}\\
\boldsymbol{v}_{t}^{U} & \ge\E\left(\left.\essinf_{U_{:t+1}^{\prime}\in\mathcal{U}_{:t+1},\,U_{:t}^{\prime}=U_{:t}}\boldsymbol{v}_{t+1}^{U^{\prime}}\right|X_{:t}\right).
\end{align*}
Equality holds, if $\mathcal{U}$ is decomposable.
\end{cor}

\begin{proof}
The assertion is immediate by combining the defining equations~\eqref{eq:v-4}
and~\eqref{eq:V-4} and the assertions of Theorem~\ref{thm:8}.
\end{proof}
\begin{rem}[Dynamic relations, pathwise description]
\label{rem:19}It holds that 
\begin{align}
V_{t}(x_{:t},u_{:t-1}) & \ge\inf_{u_{t}}\E_{X_{t+1}}\left(\left.V_{t+1}(x_{:t},X_{t+1},u_{:t})\right|X_{:t}=x_{:t},U_{:t-1}=u_{:t-1}\right)\quad\text{ and}\label{eq:11}\\
v_{t}(x_{:t},u_{:t}) & \ge\E_{X_{t+1}}\left(\left.\inf_{u_{t+1}}v_{t+1}(x_{:t},X_{t+1},u_{:t+1})\right|X_{:t}=x_{:t},U_{:t}=u_{:t}\right).\nonumber 
\end{align}
Equality holds, if $\mathcal{U}$ is decomposable.
\end{rem}

\subsection{Verification theorems}

Verification theorems provide optimality conditions. Given these characterizations
it is the purpose of verification theorems to allow verifying or checking,
if a given policy is optimal or not. An interesting, early reference
is \citet{Rockafellar1976}, who study martingales associated with
optimality conditions. \citet{Fleming1993} give verification theorems
for dynamic (in particular Markovian) problems in continuous time.
We shall address this particular situation further in more detail
below. 

The value process $\boldsymbol{v}_{t}^{U}$ is a stochastic process
depending on an underlying policy $U$. A special situation occurs
if the underlying policy $U$ is optimal, i.e., $U$ solves the initial
problem~\eqref{eq:v0}. In what follows we examine this situation.
We further provide a useful characterization of the optimizers of~\eqref{eq:v0},
relating the different concepts regarding optimization and probability
theory.

\begin{thm}[Verification theorem]
\label{thm:Verification}Let $u\in\mathcal{U}$ be any policy. Then
the stochastic processes
\[
\boldsymbol{v}_{t}^{U}=v_{t}\big(X_{:t},u_{:t}(X_{:t})\big),\qquad t=0,\dots,T,
\]
and 
\[
\boldsymbol{V}_{t}^{U}=V_{t}\big(X_{:t},u_{:t-1}(X_{:t-1})\big),\qquad t=0,\dots,T
\]
are submartingales. They are martingales, if $\mathcal{U}$ is decomposable
and if $u$ solves the initial problem~\eqref{eq:v0}.

Conversly if $\mathcal{U}$ decomposable and $\boldsymbol{V}_{t}^{U}$,
$\boldsymbol{v}_{t}^{U}$ are martingals, then $U$ is an optimizer
of~\eqref{eq:v0}. 
\end{thm}

\begin{proof}
The first assertion is immediate from Corollary~\ref{cor:Dynamic}.
For the second assume that $\boldsymbol{V}_{t}^{U^{\ast}}$, $\boldsymbol{v}_{t}^{U^{\ast}}$
are martingals for an underlying policy $U^{\ast}$. By employing~\eqref{eq:V0inf},~\eqref{eq:7-2}
and Theorem~\ref{thm:8} it follows that
\[
\inf_{U\in\mathcal{U}}\E v(X,U)=\boldsymbol{V}_{0}=\boldsymbol{V}_{0}^{U^{\ast}}=\E\left(\left.\boldsymbol{V}_{1}^{U^{\ast}}\right|X_{0}\right)=\boldsymbol{v}_{0}^{U^{\ast}}=\E\left(\left.\boldsymbol{v}_{T}^{U^{\ast}}\right|X_{0}\right)=\E\left(v\left(X,U^{\ast}\right)\right)
\]
 and thus the assertion.
\end{proof}
Theorem~\ref{thm:Verification} allows identifying a policy $U=u(X)$
as optimal policy by checking, if the value processes constitute a
martingale or not. Note that the verification theorem does not give
a hint on where and how to improve the policy. Instead, it can be
used ex post to check an existing, given policy with respect to optimality.

The verification theorem presented above notably works for every multistage
stochastic optimization problem. We did not impose other conditions
on the function $v$ except Hölder continuity, and we did not restrict
the analysis to Markovian processes. From this mathematical perspective
the statement is rather general.

\section{Specific objective functions\label{sec:Additive}}

Most common in optimal control, finance and reinforcement learning
are value functions, which accumulate costs occurring at consecutive
stages. We derive their intermediate value functions explicitly by
exploiting the specific structure of the objective function. To this
end we transform the equations for the general additive case first
and derive the equations for MDP (Markov decision processes) subsequently.
The Markovian property, from probabilistic perspective, is essential
for the MDP equations. As well, we derive the equations for stochastic
dual dynamic programming (SDDP) from the general equations.

\subsection{Lag-$\ell$ stochastic processes and additive objective functions}

The particular value function which we consider here, 
\begin{equation}
v(x_{:T},u_{:T})\coloneqq\sum_{t=1}^{T}\gamma^{t-1}\,c_{t}(x_{t-\ell:t},u_{t-\ell:t-1}),\label{eq:50}
\end{equation}
adds consecutive costs at lag $\ell\ge0$ (entries with negative stage
indices are ignored, as $u_{-1:2}=u_{0:2}$, e.g., and the corresponding
cost function is adjusted accordingly). The value function~\eqref{eq:50}
is of fundamental importance in finance and in reinforcement learning,
where $c_{t}$ is the cost associated with time~$t$ and $\gamma$
is a discount factor. Note the very particular choice of arguments
of the function~$c_{t}$: the last input element is the observation
$x_{t}$, but the subsequent decision~$u_{t}$ is not taken into
account. Figure~\ref{fig:MDP} depicts the support of the cost component
$c_{t}$ at stage $t$ (compare with Figure~\ref{fig:Progress}). 

\begin{figure}[H]
\[
\dots\leadsto\underbrace{X_{t-\ell}\leadsto u_{t-\ell}\leadsto\dots\leadsto X_{t-1}\leadsto u_{t-1}\leadsto X_{t}}_{c_{t}}\leadsto\dots
\]

\caption{Arguments of the cost component~$c_{t}$\label{fig:MDP}}
\end{figure}
The parameter $\gamma\in(-1,1)$ in~\eqref{eq:50} is most typically
interpreted as \emph{discount factor}. To derive the dynamic equations
we assume that the functions~$c_{t}$ are Hölder continuous and assume
that the stochastic process associated with the value function~\eqref{eq:Symbolic}
has lag $\ell$ as well; that is, $\sigma\left(X_{1},\dots,X_{t}\right)=\sigma\left(X_{t-\ell},\dots,X_{t}\right)$
for all $t=\ell,\dots,T$. Define the functions~$\tilde{V}_{t}$
by 
\begin{align*}
\tilde{V}_{t}(x_{:t},u_{:t-1})\cdot\gamma^{t} & \coloneqq V_{t}(x_{:t},u_{:t-1})-\sum_{i=1}^{t}\gamma^{i-1}\,c_{i}(x_{i-\ell:i},u_{i-\ell:i-1})
\end{align*}
so that $\tilde{V}_{0}=V_{0}$. For additive cost functions, the schematic
decomposition~\eqref{eq:Symbolic} now is 
\[
\inf_{u_{0}}c_{0}+\E_{X_{1}}\inf_{u_{1}}c_{1}+\cdots+\E_{X_{t}}\underbrace{\inf_{u_{t}}c_{t}+\underbrace{\E_{X_{t+1}}\inf_{u_{t+1}}c_{t+1}+\cdots+\E_{X_{T}}\inf_{u_{T}}c_{T}}_{\tilde{v}_{t}(x_{1:t},u_{0:t})}}_{\tilde{V}_{t}(x_{1:t},u_{0:t-1})}.
\]
From~\eqref{eq:v-3} we conclude that 
\begin{equation}
\tilde{V}_{t}(x_{:t},u_{:t-1})=\inf_{u_{t:}(\cdot)}\E\left(\left.\sum_{i=t+1}^{T}\gamma^{i-1-t}\,c_{i}\bigl(X_{i-\ell:i},u_{i-\ell:i-1},u_{i:T}(X_{:T})\bigr)\right|\begin{array}{l}
X_{:t}=x_{:t},\\
U_{:t-1}=u_{:t-1}
\end{array}\right).\label{eq:57}
\end{equation}
The function inside the expectation is independent of $x_{:t-\ell}$
and the stochastic process~$X$ has lag~$\ell$. Further, the decision
process~$U$ is adapted to~$X$ (cf.~\eqref{eq:v0}) and thus has
lag~$\ell$ as well. With that it follows that~\eqref{eq:57} actually
is 
\[
\tilde{V}_{t}(x_{t-\ell+1:t},u_{t-\ell+1:t-1})=\inf_{u_{t:}(\cdot)}\E\left(\left.\sum_{i=t+1}^{T}\gamma^{i-1-t}\,c_{i}\bigl(X_{:i},u_{i-\ell:i-1},u_{i:T}(X_{:T})\bigr)\right|\begin{array}{l}
X_{t-\ell+1:t}=x_{t-\ell+1:t},\\
U_{t-\ell+1:t-1}=u_{t-\ell+1:t-1}
\end{array}\right).
\]
Employing Remark~\ref{eq:11} we deduce the recursion 
\begin{align}
\MoveEqLeft[2]\tilde{V}_{t}(x_{t-\ell+1:t},u_{t-\ell+1:t-1})\label{eq:54}\\
 & \ge\inf_{u_{t}}\E_{X_{t+1}}\left(\left.\begin{array}{l}
c_{t+1}(x_{t+1-\ell:t},X_{t+1},u_{t+1-\ell:t})\\
\quad+\gamma\,\tilde{V}_{t+1}(x_{t-\ell+2:t},X_{t+1},u_{t-\ell+2:t})
\end{array}\right|\begin{array}{l}
X_{t-\ell+1:t}=x_{t-\ell+1:t},\\
U_{t-\ell+1:t-1}=u_{t-\ell+1:t-1}
\end{array}\right),\nonumber 
\end{align}
where equality indicates optimality. This backwards recursion leads
to the following discussion on MDP.

\subsection{MDP}

A Markov decision process (MDP) is a discrete-time stochastic control
process. To this end we consider the cost functions~\eqref{eq:50}
with lag $\ell=1$, i.e.,
\begin{equation}
v(x_{:T},u_{:T})\coloneqq\sum_{t=1}^{T}\gamma^{t-1}\,c_{t}(x_{t-1},x_{t};u_{t-1})\label{eq:51-2}
\end{equation}
and a process $X$ with same lag $\ell=1$, i.e., a Markovian process.
With that, the recursion~\eqref{eq:54} collapses further to 
\begin{equation}
\tilde{V}_{t}(x_{t})=\inf_{u_{t}}\E\left(\left.c_{t+1}(x_{t},X_{t+1},u_{t})+\gamma\,\tilde{V}_{t+1}(X_{t+1})\right|X_{t}=x_{t}\right).\label{eq:51-1}
\end{equation}
This recursion is well-known in MDP and~\eqref{eq:51-1} is also
known as backward induction involving the Bellman principle (cf.\ \citet{Bellman,Bellman1961}),
which is of fundamental importance in dynamic programming. 
\begin{rem}
The MDP literature considers rather trajectories which are driven
themselves by the control~$u$ (the control is called \emph{action}
in the MDP literature). To recognize this dependency in addition we
can restate the recursion as 
\[
\tilde{V}_{t}(x_{t})=\inf_{u_{t}}\E_{u_{t}}\left(\left.c_{t+1}(x_{t},X_{t+1},u_{t})+\gamma\,\tilde{V}_{t+1}(X_{t+1})\right|X_{t}=x_{t}\right),
\]
where $\E_{u}$ is the expectation with respect to the kernel $P_{u}(\cdot\mid x_{t})$,
which explicitly depends on the decision~$u$.
\end{rem}

\subsection{Dynamic optimization and Bellman's principle of optimality}

The cost function~\eqref{eq:51-2} is also considered on an infinite
horizon, i.e., 
\begin{equation}
v(x_{:T},u_{:T})\coloneqq\sum_{t=1}^{\infty}\gamma^{t-1}\,c_{t}(x_{t-1},x_{t};u_{t-1});\label{eq:51-2-1}
\end{equation}
problems in reinforcement learning are of this particular form~\eqref{eq:51-2-1}.
The value function~\eqref{eq:50} is bounded in the chosen setting,
if the cost functions are uniformly bounded, $|c_{t}|\le K<\infty$
and learning rate $\gamma\in(-1,1)$ (although most typical is $\gamma\in(0,1)$). 

A particularly interesting situation arises for cost functions which
do not depend on the stage~$t$, i.e., $c_{t}=c$ and decision satisfying
$(X_{t},X_{t+1})\sim(X,X^{\prime})$. Then, the value functions $\tilde{V}_{t}$
does not depend on~$t$ neither and the equation
\begin{equation}
\tilde{V}(x)=\inf_{u}\E\left(\left.c(x,X^{\prime},u)+\gamma\,\tilde{V}(X^{\prime})\right|X=x\right)\label{eq:15}
\end{equation}
holds. 

This is a fixed point equation and Banach's fixed point theorem can
be applied to prove existence and uniqueness of the value function~$\tilde{V}$
in appropriate spaces. As well, the equation~\eqref{eq:15} specifies
an iterative scheme to improve the value function~$\tilde{V}$ in
consecutive steps. As an example we state the following, where we
refer to \citet{FletenPichler} for a proof in a similar situation.
\begin{thm}
Suppose that $c$ is continuous and $X\in K$ a.s.\ for some compact
set $K\subset\mathbb{R}^{n}$ and $|\gamma|<1$. Then the value function
$\tilde{V}$ is continuous and $\tilde{V}\in C(K)$.
\end{thm}

\subsection{SDDP}

The problem setting of stochastic dual dynamic programming (SDDP)
considers a stagewise independent stochastic process $X_{t}$ (i.e.,
$X_{t}$ is \emph{independent} of all preceding~$X_{t^{\prime}}$,
$t^{\prime}<t$), which is a further simplification of all situations
described above. With $X_{t}\sim X$, the dynamic equation reduces
further to
\begin{equation}
\tilde{V}_{t}(x_{t})=\inf_{u_{t}}\E\left(c_{t+1}(x_{t},X_{t+1},u_{t})+\gamma\,\tilde{V}_{t+1}(X_{t+1})\right).\label{eq:SDDP}
\end{equation}

This is the simplest situation from statistic perspective and it is
not surprising that large and extensive problem settings are accessible
for numerical computations. The important algorithm for SDDP for solving
the problem~\eqref{eq:SDDP} efficiently originated in \citet{PereiraSDDP}. 

We refer to \citet{Shapiro2010} for an extended analysis of the algorithm,
to \citet{Romisch} and to \citet{Leclere2014,Philpott2008450} for
convergence proofs of the algorithm. 

\section{Summary}

Multistage stochastic optimization has many applications in varying
areas, from finance to data science to just mention two. The problems
are popular and typically stated conditioned on partial realizations.
This pathwise, or $\omega$\nobreakdash-by\nobreakdash-$\omega$,
perspective lacks mathematical rigor. It is surprising that mathematical
foundations regarding measurability are incomplete from a mathematical
perspective and still missing. 

This paper clarifies that multistage optimization problems, even if
given in an informal, pathwise or $\omega$\nobreakdash-by\nobreakdash-$\omega$
way can be cast with mathematical rigor. We start by outlining the
general problem and employ the Kolmogorov continuity theorem to verify
that value functions are well defined, even if conditioned on sets
of measure zero. 

Verification theorems can be employed to confirm that candidate policies
are optimal. We further characterize optimal policies by involving
martingales to characterize these optimal solutions. 

Markov decision processes, the Bellman principle for reinforcement
learning and stochastic dual dynamic programming are probably most
well-known and common in practice of dynamic programming. We derive
these problem settings as special cases and, in this way, provide
rigorous mathematical foundations.

\bibliographystyle{abbrvnat}
\bibliography{../../Literatur/LiteraturAlois,LiteraturPaul}

\end{document}